\documentclass{amsart}

\usepackage{amsmath,amssymb,amsfonts}

\usepackage{hyperref}

\newcommand{\X}{\mathcal{X}}
\newcommand{\arr}{\longrightarrow}

\newcommand{\G}{\mathfrak{G}}
\newcommand{\wt}{\widetilde}
\newcommand{\be}{\mathsf{o}}
\newcommand{\en}{\mathsf{t}}

\newcommand{\mapdown}[1]%
{\Big\downarrow\rlap{$\vcenter{\hbox{$\scriptstyle#1$}}$}}

\newcommand{\N}{\mathbb{N}}
\newcommand{\R}{\mathbb{R}}
\newcommand{\Z}{\mathbb{Z}}

\newtheorem{theorem}{Theorem}[section]
\newtheorem{lemma}[theorem]{Lemma}
\newtheorem{proposition}[theorem]{Proposition}

\theoremstyle{definition}
\newtheorem{defi}{Definition}[section]
\newtheorem{examp}{Example}[section]

\title{Growth of \'etale groupoids and simple algebras}
\author{Volodymyr Nekrashevych}

\begin{document}

\maketitle

\begin{abstract}
We study growth and complexity of \'etale groupoids in relation to
growth of their convolution algebras. As an application, we construct
simple finitely generated algebras of arbitrary Gelfand-Kirillov
dimension $\ge 2$ and simple finitely generated algebras of quadratic growth over
arbitrary fields.
\end{abstract}

\tableofcontents

\section{Introduction}

Topological groupoids are extensively used in dynamics, topology,
non-commutative geometry, and
$C^*$-algebras,
see~\cite{haefl:foliations,paterson:gr,renault:groupoids}. 
With recent results on topological full groups (see~\cite{matui:etale,juschenkomonod,YNS})
new applications of groupoids to group theory were
discovered. 

Our paper studies growth and complexity for \'etale groupoids
with applications to the theory of growth and Gelfand-Kirillov
dimension of algebras. We give examples of groupoids
whose convolution algebras (over an arbitrary field) have prescribed growth.
In particular, we give first examples of simple algebras of quadratic
growth over finite fields and simple algebras of Gelfand-Kirillov
dimension 2 that do not have quadratic growth. 

A \emph{groupoid} $\G$ is the set of isomorphisms of a small category,
i.e., a set $\G$ with partially defined multiplication
and everywhere defined operation of taking inverse satisfying the
following axioms:
\begin{enumerate}
\item If the products $ab$ and $bc$ are defined, then $(ab)c$ and $a(bc)$ are
  defined and are equal.
\item The products $a^{-1}a$ and $bb^{-1}$ are always defined and
  satisfy $abb^{-1}=a$ and $a^{-1}ab=b$ whenever the product $ab$ is defined.
\end{enumerate}
It follows from the axioms that $(a^{-1})^{-1}=a$ and that a product
$ab$ is defined if and only if $bb^{-1}=a^{-1}a$.
The elements of the form $aa^{-1}$ are called \emph{units} of the
groupoid. We call $\be(g)=g^{-1}g$ and $\en(g)=gg^{-1}$ the \emph{origin}
and the \emph{target} of the element $g\in\G$.

A \emph{topological groupoid} is a groupoid together with topology
such that multiplication and taking inverse are continuous.
It is called \emph{\'etale} if every element has a basis of neighborhoods
consisting of \emph{bisections}, i.e., sets $F$ such that
$\be:F\arr\be(F)$ and $\en:F\arr\en(F)$ are homeomorphisms. 

For example, if $G$ is a discrete group acting (from the left) 
by homeomorphisms on a topological space
$\X$, then the topological space $G\times\X$ has a natural structure of an \'etale
groupoid with respect to the multiplication
\[(g_1, g_2(x))\cdot (g_2, x)=(g_1g_2, x).\]
In some sense \'etale groupoids are generalization of actions of
discrete groups on topological spaces.

We consider two growth functions for an \'etale
groupoid $\G$ with compact totally disconnected
space of units. The first one is the most straightforward and classical: growth of
fibers of the origin map. If $S$ is an open compact generating set of
$\G$ then, for a given unit $x$, we can consider the growth function $\gamma_S(r,
x)$ equal to the 
number of groupoid elements with origin in $x$ that can be expressed
as a product of at most $n$ elements of $S\cup S^{-1}$. This notion of
growth of a groupoid has appeared in many situations, especially in
amenability theory for topological groupoids,
see~\cite{kaim,delaroche_renault}. 
See also Theorem~\ref{th:nofree} of our paper,
where for a class of groupoids we show how sub-exponential growth
implies absence of free subgroups in the topological full group of the
groupoid.

This notion of growth does not capture full complexity of a groupoid
precisely because it is ``fiberwise''. Therefore, we introduce the
second growth function: complexity of the groupoid. Let $\mathcal{S}$
be a finite covering by open bisections of an open
compact generating set $S$ of $\G$. For a given natural number $r$
and units $x, y\in\G^{(0)}$ we write $x\sim_r y$ if for any two
products $S_1S_2\ldots S_n$ and $R_1R_2\ldots R_m$ of elements of
$\mathcal{S}\cup\mathcal{S}^{-1}$ such that $n, m\le r$ we have $S_1S_2\ldots
S_nx=R_1R_2\ldots R_mx$ if and only if $S_1S_2\ldots S_ny=R_1R_2\ldots
R_my$. In other words, $x\sim_r y$ if and only if balls of radius $r$
with centers in $x$ and $y$
in the natural $\mathcal{S}$-labeled \emph{Cayley graphs} of
$\G$ are isomorphic. Then the \emph{complexity function} $\delta(r, \mathcal{S})$ is
the number of $\sim_r$-equivalence classes of points of $\G^{(0)}$.

This notion of complexity (called in this case \emph{factor complexity}, or
\emph{subword complexity}) is well known and studied for groupoids of
the action of shifts on closed
shift-invariant subsets of $X^{\Z}$, where $X$ is a finite
alphabet. There is an extensive literature on it,
see~\cite{cassaigneets:factorcomplexity,ferenczi:complexity}
An interesting result from the group-theoretic point of view is a
theorem of
N.~Matte~Bon~\cite{mattebon:liouville} stating
that if complexity of a subshift is
strictly sub-quadratic, then the topological full group of the
corresponding groupoid is Liouville. Here the \emph{topological full
group} of an \'etale groupoid $\G$ is the group of all
$\G$-bisections $A$ such that $\be(A)=\en(A)=\G^{(0)}$.

It seems that complexity of groupoids in more general
\'etale groupoids has not been well studied yet. It would be
interesting to understand how complexity function (together with the
growth of fibers) is related with the properties of the topological
full group of an \'etale groupoid. For example, it would be
interesting to know if there exists  a non-amenable (e.g., free) group
acting faithfully on a compact topological space so that the
corresponding groupoid of germs has sub-exponential growth and
sub-exponential complexity functions.

We relate growth and complexity of groupoids with growth of algebras
naturally associated with them. Suppose that $\mathcal{A}$ is a finitely
generated algebra with a unit over a field $\Bbbk$. Let
$V$ be the $\Bbbk$-linear span of
a finite generating set containing the unit. Denote by $V^n$ the linear
span of all products $a_1a_2\ldots a_n$ for $a_i\in V$. Then
$\mathcal{A}=\bigcup_{n=1}^\infty V^n$. \emph{Growth} of $\mathcal{A}$
is the function
\[\gamma(n)=\dim V^n.\]
It is easy to see that if $\gamma_1, \gamma_2$ are growth functions
defined using different finite generating sets, then there exists $C>1$ such
that $\gamma_1(n)\le\gamma_2(Cn)$ and $\gamma_2(n)\le\gamma_1(Cn)$.

\emph{Gelfand-Kirillov} dimension of $\mathcal{A}$ is defined as
$\limsup_{n\to\infty}\frac{\log\dim V^n}{\log n}$, which informally is
the degree of polynomial growth of the algebra. If $\mathcal{A}$ is
not finitely generated, then its Gelfand-Kirillov dimension is defined
as the supremum of the Gelfand-Kirillov dimensions of all its sub-algebras.
See the monograph~\cite{krauselenagan} for a survey of results on growth of algebras
and their Gelfand-Kirillov dimension.

It is known, see~\cite{warfield:gk}
and~\cite[Theorem~2.9]{krauselenagan}, that Gelfand-Kirillov dimension can be any 
number in the set $\{0, 1\}\cup [2, \infty]$. The values in the interval
$(1, 2)$ are prohibited by a theorem of G.M.~Bergman,
see~\cite[Theorem~2.5]{krauselenagan}.
There are examples of prime algebras of
arbitrary Gelfand-Kirillov dimension $d\in [2, \infty]$, see~\cite{vishne:gk}, but
it seems that no examples of simple algebras of arbitrary Gelfand-Kirillov dimension
over finite fields were known so far.

A naturally defined \emph{convolution algebra} $\Bbbk[\G]$ over
arbitrary field $\Bbbk$ is 
associated with every \'etale groupoid $\G$ with totally
disconnected space of units. If the groupoid $\G$ is Hausdorff, then
$\Bbbk[\G]$ is the convolution algebra of all continuous functions
$f:\G\arr\Bbbk$ with compact support, where $\Bbbk$ is taken with the 
discrete topology. Here convolution $f_1\cdot f_2$ of two functions 
is the function given by the formula
\[f(g)=\sum_{g_1g_2=g}f_1(g_1)f_2(g_2).\]

In the non-Hausdorff case we follow A.~Connes~\cite{conn:foliations}
and B.~Steinberg~\cite{steinberg:groupoidapproach}, and define $\Bbbk[\G]$ as the linear span of
the functions that are continuous on open compact subsets of
$\G$. Equivalently, $\Bbbk[\G]$ is the linear span of the
characteristic functions of open compact $\G$-bisections. 

Note that
the set $\mathcal{B}(\G)$ of all open compact $\G$-bisections (together with the empty
one) is a semigroup. The algebra $\Bbbk[\G]$ is isomorphic to
the quotient of the semigroup algebra of $\mathcal{B}(\G)$
by the ideal generated by the relations $F-(F_1+F_2)$ for
all triples $F, F_1, F_2\in\mathcal{B}(\G)$ such that
$F=F_1\cup F_2$ and $F_1\cap F_2=\emptyset$.

We prove the following relation between growth of groupoids and growth
of their convolution algebras.

\begin{theorem}
\label{th:main}
Let $\G$ be an \'etale groupoid with compact totally disconnected
space of units. Let $\mathcal{S}$ be a finite set of open compact
$\G$-bisections such that $S=\bigcup\mathcal{S}$ is a generating set of
$\G$. Let $V\subset\Bbbk[\G]$ be the linear span of the characteristic
functions of elements of $\mathcal{S}$. Then
\[\dim V^n\le\overline\gamma(r, \mathcal{S})\delta(r, \mathcal{S}),\]
where $\overline\gamma(r,
\mathcal{S})=\max_{x\in\G^{(0)}}\gamma_S(r, x)$.
\end{theorem}

We say that a groupoid $\G$ is \emph{minimal} if every $\G$-orbit is
dense in $\G^{(0)}$. We say that $\G$ is \emph{essentially principal}
if the set of points $x$ with trivial isotropy group
is dense in $\G^{(0)}$. Here the isotropy group of a point $x$ is the
set $\{g\in\G\;:\;\be(g)=\en(g)=x\}$. It is known,
see~\cite{brownclarketc:simplicity}, that for a Hausdorff minimal essentially principal groupoid
$\G$ with compact totally disconnected set of units the algebra
$\Bbbk[\G]$ is simple. We give a proof of this fact for completeness in
Proposition~\ref{pr:simple}. 

We give in Proposition~\ref{prop:expansive} a condition
(related to the classical notion of an \emph{expansive dynamical system})
ensuring that $\Bbbk[\G]$ is finitely generated.

Fibers of the origin map provide us with naturally defined 
$\Bbbk[\G]$-modules. Namely, for a given unit $x\in\G^{(0)}$ consider
the vector space $\Bbbk\G_x$ of functions $\phi:\G_x\arr\Bbbk$ with
finite support, where $\G_x=\be^{-1}(x)$ is the set of elements of the
groupoid $\G$ with origin in $x$. Then convolution $f\cdot \phi$ for
any $f\in\Bbbk[\G]$ and $\phi\in\Bbbk\G_x$ is an element of
$\Bbbk\G_x$, and hence $\Bbbk\G_x$ is a left $\Bbbk[\G]$-module.

It is easy to prove that if the isotropy group of $x$ is trivial, then
$\Bbbk\G_x$ is simple and that growth of $\Bbbk\G_x$ is bounded by
$\gamma_S(x, r)$, see Proposition~\ref{pr:modules}.

As an example of applications of these results, we consider the
following family of algebras. Let $X$ be a finite alphabet, and let
$w:X\arr\Z$ be a bi-infinite sequence of elements of $X$. 
Denote by $D_x$, for $x\in X$ the diagonal matrix  $(a_{i, j})_{i,
  j\in\Z}$ given by
\[a_{i, j}=\left\{\begin{array}{ll} 1 & \text{if $i=j$ and $w(i)=x$,}\\ 0 &
    \text{otherwise.}\end{array}\right.\]
Let $T$ be the matrix $(t_{i, j})_{i, j\in\Z}$ of the shift given by
\[t_{i, j}=\left\{\begin{array}{ll} 1 & \text{if $i=j+1$,}\\ 0 &
    \text{otherwise.}\end{array}\right.\]
Fix a field $\Bbbk$, and let $\mathcal{A}_w$ be the $\Bbbk$-algebra
generated by the matrices $D_x$, for $x\in X$, by $T$, and its transpose $T^\top$.

We say that $w$ is \emph{minimal} if for every finite subword
$(w(n), w(n+1), \ldots, w(n+k))$ there exists $R>0$ such that for any
$i\in\Z$ there exists $j\in\Z$ such that $|i-j|\le R$ and
$(w(j), w(j+1), \ldots, w(j+k))=(w(n), w(n+1), \ldots, w(n+k))$. We
say that $w$ is \emph{non-periodic} if there does not exist $p\ne 0$
such that $w(n+p)=w(n)$ for all $n\in\Z$.
\emph{Complexity function} $p_w(n)$ of the sequence $w\in X^{\Z}$ is
the number of different subwords $(w(i), w(i+1), \ldots, w(i+n-1))$ of
length $n$ in $w$.

The following theorem is a corollary of the results of our paper,
see Subsection~\ref{sss:subshifts} and Example~\ref{ex:matrices}.

\begin{theorem}
Suppose that $w\in X^{\Z}$ is minimal and non-periodic. Then the
algebra $\mathcal{A}_w$ is simple, and its growth $\gamma(n)$
satisfies \[C^{-1}n\cdot p_w(C^{-1}n)\le Cn\cdot p_w(Cn)\]
for some $C>1$.
\end{theorem}

We can apply now results on complexity of sequences to
construct simple algebras of various growths. For example, if $w$ is
\emph{Sturmian}, then $p_w(n)=n+1$, and hence $\mathcal{A}_w$ has quadratic growth. For
different \emph{Toeplitz} sequences we can obtain simple algebras of
arbitrary Gelfand-Kirillov dimension $d\ge 2$, or simple algebras of
growth $n\log n$, etc., see Subsection~\ref{sss:subshifts}.

Another class of examples of groupoids considered in our paper are
groupoids associated with groups acting on a rooted tree. If $G$ acts
by automorphisms on a locally finite rooted tree $T$, then it acts
by homeomorphisms on the boundary $\partial T$. One can consider the
\emph{groupoid of germs} $\G$ of the action. Convolution algebras
$\Bbbk[\G]$ are related to the \emph{thinned algebras} studied in~\cite{sid:ring,bartholdi:ring}.
In the case when $G$ is a \emph{contracting self-similar group},
Theorem~\ref{th:main} implies a result of L.~Bartholdi from~\cite{bartholdi:ring} giving
an estimate of Gelfand-Kirillov dimension for the thinned algebras of
contracting self-similar groups.

\section{\'Etale groupoids}

A \emph{groupoid} is a small category of isomorphisms (more precisely,
the set of its morphisms). For a groupoid
$\G$, we denote by $\G^{(2)}$ the set of composable pairs, i.e., the
set of pairs $(g_1, g_2)\in\G\times\G$ such that the product $g_1g_2$
is defined. We denote by $\G^{(0)}$ the set of units of $\G$, i.e.,
the set of identical isomorphisms. We also denote by $\be,
\en:\G\arr\G^{(0)}$ the \emph{origin} and \emph{target} maps given by
\[\be(g)=g^{-1}g,\qquad \en(g)=gg^{-1}.\]
We interpret then an element $g\in\G$ as an arrow from $\be(g)$ to
$\en(g)$. The product $g_1g_2$ is defined if and only if
$\en(g_2)=\be(g_1)$.

For $x\in\G^{(0)}$, denote
\[\G_x=\{g\in\G\;:\;\be(g)=x\},\qquad\G^x=\{g\in\G\;:\;\en(g)=x\}.\]

The set $\G_x\cap\G^x$ is called the \emph{isotropy group} of $x$. A
groupoid is said to be \emph{principal} (or an equivalence relation)
if the isotropy group of every point is trivial. 
Two
units $x, y\in\G^{(0)}$ belong to one \emph{orbit} if there exists
$g\in\G$ such that $\be(g)=x$ and $\en(g)=y$. It is easy to see that
belonging to one orbit is an equivalence relation.

A \emph{topological groupoid} is a groupoid $\G$ with a topology on it
such that multiplication $\G^{(2)}\arr\G$ and taking inverse
$\G\arr\G$ are continuous maps. We do not require that $\G$ is
Hausdorff, though we assume that the space of units $\G^{(0)}$ is
metrizable and locally compact.

A \emph{$\G$-bisection} is a subset $F\subset\G$ such that the maps
$\be:F\arr\be(F)$ and $\en:F\arr\en(F)$ are homeomorphisms. 

\begin{defi}
A topological groupoid $\G$ is \emph{\'etale} if the set of all open
$\G$-bisections is a basis of the topology of $\G$. 
\end{defi}

Let $\G$ be an \'etale groupoid. It is easy to see  that
product of two open bisections is an open bisection. It follows that
for every bisection $F$ the sets $\be(F)=F^{-1}F$ and $\en(F)=FF^{-1}$
are open, which in turn implies that $\G^{(0)}$ is an open subset of $\G$.

If $\G$ is not Hausdorff, then there exist $g_1, g_2\in\G$ that do not
have disjoint bisections. Since $\G^{(0)}$ is Hausdorff,
this implies that $\be(g_1)=\be(g_2)$ and $\en(g_1)=\en(g_2)$. It
follows that the unit $x=\be(g_1)$ and the element $g_2^{-1}g_1$ of
the isotropy group of $x$ do not have disjoint open neighborhoods. In
particular, it means that principal \'etale groupoids are always
Hausdorff, and that an \'etale groupoid is Hausdorff if and only if
$\G^{(0)}$ is a closed subset of $\G$.

\begin{examp}
Let $G$ be a discrete group acting by homeomorphisms on a space $\X$. Then the
space $G\times\X$ has a natural groupoid structure with given by the
multiplication
\[(g_2, g_1(x))(g_1, x)=(g_2g_1, x).\]
This is an \'etale groupoid, since every set $\{g\}\times\X$ is an
open bisection. The groupoid $G\times\X$ is called the \emph{groupoid
  of the action}, and is denoted $G\ltimes\X$.
\end{examp}

Our main class of groupoids will be naturally defined 
quotients of the groupoids of actions,
called groupoids of germs.

\begin{examp}
Let $G$ and $\X$ be as in the previous example. A
\emph{germ} is an equivalence class of a pair $(g, x)\in G\times\X$
where $(g_1, x)$ and $(g_2, x)$ are equivalent if there exists a
neighborhood $U$ of $x$ such that the maps $g_1:U\arr\X$ and
$g_2:U\arr\X$ coincide. The set of germs is also an \'etale groupoid
with the same multiplication rule as in the previous example. We call
it \emph{groupoid of germs of the action}.
\end{examp}

The spaces of units in both groupoids are naturally identified with
the space $\X$ (namely, we identify the pair or the germ $(1, x)$ with
$x$). The groupoid of the action is Hausdorff if $\X$ is Hausdorff, since it is
homeomorphic to $G\times\X$. The groupoid of germs, on the other hand,
is frequently non-Hausdorff, even for a Hausdorff space $\X$.

If every germ of every non-trivial element of $G$ is not a unit (i.e.,
not equal to a germ of the identical homeomorphism), then
the groupoid of the action coincides with the groupoid of germs.

Many interesting examples of \'etale groupoids appear in dynamics and
topology, see~\cite{haefl:foliations,bellissardjuliensavinien,nek:hyperbolic}.

\section{Compactly generated groupoids}

For the rest of the paper, $\G$ is an \'etale groupoid such that
$\G^{(0)}$ is a compact totally disconnected metrizable space.
Note that then there exists a basis of topology of $\G$
consisting of open compact $\G$-bisections. Note that we allow compact
non-closed and compact non-Hausdorff sets, since $\G$ in general is
not Hausdorff. However, if
$F$ is an open compact bisection, then $\be(F)$ and $\en(F)$ are clopen
(i.e., closed and open) and $F$ is Hausdorff.

\subsection{Cayley graphs and their growth}

\begin{defi}
A groupoid $\G$ with compact totally disconnected unit space
is \emph{compactly generated} if there exists a
open compact subset $S\subset\G$ such that $\G=\bigcup_{n\ge 0}(S\cup S^{-1})^n$. 
The set $S$ is called the \emph{generating set} of $\G$.
\end{defi}

This definition is equivalent (for \'etale groupoids with compact
totally disconnected unit space) to the definition of~\cite{haefliger:compactgen}.

\begin{examp}
Let $G$ be a group acting on a Cantor set
$\X$. If $S$ is a finite generating set of $G$, then $S\times\X$ is an
open compact generating set of the groupoid $G\ltimes\X$. The set all
of germs of elements of $S$ is an open compact generating set of the
groupoid of germs of the action. Thus, both groupoids are compactly
generated if $G$ is finitely generated.
\end{examp}

Let $S$ be an open compact generating set of $\G$. Let
$x\in\G^{(0)}$. The \emph{Cayley graph} $\G(x, S)$ is the directed
graph with the set of vertices $\G_x$ in which we have an arrow from
$g_1$ to $g_2$ whenever there exists $s\in S$ such that $g_2=sg_1$.

We will often consider the graph $\G(x, S)$ as a \emph{rooted graph} with root $x$. Morphism
$\phi:\Gamma_1\arr\Gamma_2$ of
rooted graphs is a morphism of graphs that maps the root of $\Gamma_1$
to the root of $\Gamma_2$.

Note that since $S$ can be covered
by a finite set of bisections, the degrees of vertices of the graphs
$\G(x, S)$ are uniformly bounded.

\begin{examp}
Let $G$ be a finitely generated group acting on a totally disconnected
compact space $\X$. Let $S$ be a finite generating set of $G$, and let
$S\times\X$ be the corresponding generating set of the groupoid of
action $G\ltimes\X$. The
Cayley graphs $G\ltimes\X(x, S\times\X)$ coincide then with the Cayley
graphs of $G$ (with respect to the generating set $S$).

The groupoid of germs $\G$ will have smaller Cayley graphs. Let
$S'\subset\G$ be the set of all germs of elements of $S$. Denote,
for $x\in\X$, by $G_{(x)}$ the subgroup of $G$ consisting of all
elements $g\in G$ such that there exists a neighborhood $U$ of $x$
such that $g$ fixes every point of $U$. Then $\G(x, S')$ is isomorphic
to the \emph{Schreier graph} of $G$ modulo $G_{(x)}$. Its vertices are
the cosets $hG_{(x)}$, and a coset $h_1G_{(x)}$ is connected by an
arrow with $h_2G_{(x)}$ if there exists a generator $s\in S$ such that
$sh_1G_{(x)}=h_2G_{(x)}$.
\end{examp}

Cayley graphs $\G(x, S)$ are closely related to the \emph{orbital
graphs}, which are defined as graphs $\Gamma(x, S)$ with the set of
vertices equal to the orbit of $x$, in which a vertex $x_1$ is
connected by an arrow to a vertex $x_2$ if there exists $g\in S$ such
that $\be(s)=x_1$ and $\en(s)=x_2$. Orbital graph $\Gamma(x, S)$ is the
quotient on the Cayley graph $\G(x, S)$ by the natural right action of
the isotropy group of $x$. In particular, orbital graph and the Cayley
graph coincide if the isotropy group of $x$ is trivial.

Denote by $B_S(x, n)$ the ball of radius $n$ with center $x$ in the
graph $\G(x, S)$ seen as a rooted graph (with root $x$). Let
\[\gamma_S(x, n)=|B_S(x, n)|,\qquad
\overline\gamma(n, S)=\max_{x\in\G^{(0)}}\gamma_S(x, n).\]

If $S_1$ and $S_2$ are two open compact generating sets of $\G$, then
there exists $m$ such that $S_2\subset\bigcup_{1\le k\le m}(S_1\cup
S_1^{-1})^k$ and $S_1\subset\bigcup_{1\le k\le m}(S_2\cup
S_2^{-1})^k$. Then $\gamma_{S_1}(x, mn)\ge\gamma_{S_2}(x,
n)$ and $\gamma_{S_2}(x, mn)\ge\gamma_{S_1}(x, n)$ for all $n$. It also follows
that $\overline\gamma(mn, S_1)\ge\overline\gamma(n, S_2)$ and
$\overline\gamma(mn, S_2)\ge\overline\gamma(n, S_1)$ for all $n$. In
other words, the \emph{growth rate} of the functions $\gamma_S(x, n)$
and $\overline\gamma(n, S)$ do not depend on the choice of $S$, if $S$
is a generating set.

Condition of polynomial growth of Cayley graphs of groupoids (or, in
the measure-theoretic category, of connected components of graphings of
equivalence relations) appear in the study of amenability of
groupoids, see~\cite{kaim,delaroche_renault}.

Here is another example of applications of the notion of growth of groupoids.

\begin{theorem}
\label{th:nofree}
Let $G$ be a finitely generated subgroup of the automorphism group of a locally finite
rooted tree $T$. Consider the groupoid of germs $\G$ of the action of
$G$ on the boundary $\partial T$ of the tree. If $\gamma_S(x, n)$
has sub-exponential growth for every $x\in\partial T$, then $G$ has no
free subgroups.
\end{theorem}

\begin{proof}
By~\cite[Theorem~3.3]{nek:free}, if $G$ has a free subgroup, then either there exists a free
subgroup $F$ and a point $x\in\partial T$ such that the stabilizer of
$x$ in $F$ is trivial, or there exists a free subgroup $F$ and a point
$x\in\partial T$ such that $x$ is fixed by $F$ and every non-trivial
element $g$ of $F$ the germ $(g, x)$ is non-trivial. But both
conditions imply that the Cayley graph $\G(x, S)$ has exponential growth.
\end{proof}

\subsection{Complexity}
Let $\mathcal{S}$ be a finite set of open compact
$\G$-bisections such that $S=\bigcup\mathcal{S}$ is a generating set. Note that
every compact subset of $\G$ can be covered by a finite number of open
compact $\G$-bisections.
 
Denote by $\G(x, \mathcal{S})$ the oriented labeled
graph with the set of vertices $\G_x$ in which we have an arrow
from $g_1$ to $g_2$ labeled by $A\in\mathcal{S}$ if there exists $s\in
A$ such that $g_2=sg_1$. 

The graph $\G(x, \mathcal{S})$ basically
coincides with $\G(x, S)$ for $S=\bigcup\mathcal{S}$. The only difference is the
labeling and that some
arrows of $\G(x, S)$ become multiple arrows in $\G(x,
\mathcal{S})$. In particular, the metrics induced on the sets of
vertices of graphs $\G(x, S)$ and $\G(x, \mathcal{S})$ coincide.

We denote by $B_{\mathcal{S}}(x, r)$ or just by $B(x, r)$ the ball of radius $r$ with center in $x$, seen
as a rooted oriented labeled graph.
We write $x\sim_r y$ if $B_{\mathcal{S}}(x, r)$ and $B_{\mathcal{S}}(y, r)$ are isomorphic.

\begin{defi}
\emph{Complexity} of $\mathcal{S}$ is the function $\delta(r,
\mathcal{S})$ equal to the number of $\sim_r$-equivalence classes.
\end{defi}

It is easy to see that $\delta(r, \mathcal{S})$ is finite for every
$r$ and $\mathcal{S}$.

\subsection{Examples}

\subsubsection{Shifts}
\label{sss:shifts}
Let $X$ be a finite alphabet containing more than one letter. Consider
the space $X^{\Z}$ of all bi-infinite words over $X$, i.e., maps
$w:\Z\arr X$. Denote by $s:X^{\Z}\arr X^{\Z}$ the shift map given by the rule
$s(w)(i)=w(i+1)$. The space $X^{\Z}$ is homeomorphic to the Cantor set
with respect to the direct product topology (where $X$ is
discrete).

A \emph{sub-shift} is a closed $s$-invariant subset
$\X\subset X^{\Z}$. We always assume that $\X$ has no isolated points.
For a sub-shift $\X$, consider the groupoid
$\mathfrak{S}$ of the germs of the 
action of $\Z$ on $\X$ generated by the shift. It is easy to see that
all germs of non-zero powers of the shift are non-trivial,
hence the groupoid $\mathfrak{S}$ coincides with the
groupoid $\Z\ltimes\X$ of the action.
As usual, we will identify $\X$ with the space of units $\mathfrak{S}^{(0)}$. The
set $S=\{(s, x)\;:\;x\in\X\}$ is an open compact generating set of $\mathfrak{S}$. The
Cayley graphs $\mathfrak{S}(w, S)$ are isomorphic to the Cayley graph of $Z$
with respect to the generating set $\{1\}$.

If $\X$ is \emph{aperiodic}, i.e., if it does not contain periodic
sequences, then $\mathfrak{S}$ is principal. Note that $\mathfrak{S}$ is always Hausdorff.

For $x\in X$, denote by $S_x$ set of germs of the restriction of $s$
onto the cylindrical set $\{w\in\X\;:\;w(0)=x\}$. Then $\mathcal{S}=\{S_x\}_{x\in
  X}$ is a covering of $S$ by disjoint clopen subsets of $S$. Then for
every $w\in\X$, the Cayley graph $\mathfrak{S}(w, \mathcal{S})$ basically
repeats $w$: its set of vertices is the set of germs $(s^n, w)$, $n\in
\Z$; for every $n$ we have an arrow from $(s^n, w)$ to $(s^{n+1}, w)$
labeled by $S_{w(n)}$.

In particular, we have 
\[\delta(n, \mathcal{S})=p_{\X}(2n),\]
where $p_{\X}(k)$ denotes the
number of words of length $k$ that appear as subwords of elements of
$\X$.

Complexity $p_{\X}(n)$ of subshifts is a well studied subject,
see~\cite{kurka:topsymb,ferenczi:complexity,cassaigneets:factorcomplexity}
and references therein.

Two classes of subshifts are especially interesting for us: Sturmian and
Toeplitz subshifts.

Let $\theta\in (0, 1)$ be an irrational number, and consider 
the rotation \[R_\theta:x\mapsto x+\theta\pmod{1}\] of the circle
$\R/\Z$. For a number $x\in\R/\Z$ not belonging to the
$R_\theta$-orbit of $0$, consider the
\emph{$\theta$-itinerary $I_{\theta, x}\in\{0, 1\}^{\Z}$} given by
\[I_{\theta, x}(n)=\left\{\begin{array}{ll} 0 & \text{if $x+n\theta\in (0,
      \theta)\pmod{1}$},\\
1 & \text{if $x+n\theta\in (\theta, 1)\pmod{1}$}.\end{array}\right.\]
In other words, $I_{\theta, x}$ describes the itinerary of $x\in\R/\Z$ under
the rotation $R_\theta$ with respect to the partition
$[0, \theta), [\theta, 1)$ of the circle $\R/\Z$. 
If $x$ belongs to the orbit of $0$, then we define two itineraries
$I_{\theta, x+0}=\lim_{t\to x+0}I_{\theta, t}$ and
$I_{\theta, x-0}=\lim_{t\to x-0}I_{\theta, t}$, where $t$ in the limits
belongs to the complement of the orbit of $0$.

The set $\X_\theta$ of all itineraries is a subshift of $\{0, 1\}^{\Z}$ called the
\emph{Sturmian subshift} associated with $\theta$. Informally, the space
$\X_\theta$ is obtained from the circle $\R/\Z$ by ``cutting'' it
along the $R_\theta$-orbit of $0$, i.e., by replacing each point
$x=n\theta$ by two copies $x+0$ and $x-0$. A basis of topology of
$\X_\theta$ is the set of arcs of the form $[n\theta+0,
m\theta-0]$. The shift is identified
in this model with the natural map induced by the
rotation $R_\theta$. 

Complexity $p_{\X_\theta}(n)$ of the Sturmian subshift is equal to the number of all possible
$R_\theta$-itineraries of length $n$. Consider the set
$\{R_\theta^{-k}(\theta)\}_{k=0, 1, \ldots, n}$. It separates the
circle $\R/\Z$ into $n+1$ arcs such that two points $x, y$ have equal
length $n$ segments $\{0, \ldots, n-1\}\arr\{0, 1\}$ of their
itineraries $I_{\theta, x}$, $I_{\theta, y}$ if and only if
they belong to one arc. It follows that $p_{\X_\theta}(n)=n+1$. The
subshifts of the form $\X_\theta$ and their elements are called
\emph{Sturmian} subshifts and \emph{Sturmian} sequences.

A sequence $w:X\arr\Z$ is a \emph{Toeplitz} sequence if it is not
periodic and for every
$n\in\Z$ there exists $p\in\N$ such that $w(n+kp)=w(n)$ for all
$k\in\Z$. Complexity of Toeplitz sequences is well studied.

It is known, for example,
(see~\cite[Proposition~4.79]{kurka:topsymb}) that for any
$1\le\alpha\le\beta\le\infty$ there exists a Toeplitz subshift $\X$
(i.e., closure of the shift orbit of a Toeplitz sequence)
such that
\[\liminf_{n\to\infty}\frac{\ln p_{\X}(n)}{\ln n}=\alpha,\qquad
\limsup_{n\to\infty}\frac{\ln p_{\X}(n)}{\ln n}=\beta.\]

The following theorem is proved by M.~Koskas in~\cite{koskas}.

\begin{theorem}
For every rational number $p/q>1$
and every positive increasing differentiable function $f(x)$ satisfying
$f(n)=o(n^\alpha)$ for all $\alpha>0$, and $nf'(n)=o(n^\alpha)$ for
all $\alpha>0$, there exists a Toeplitz subshift $\X$ and two
constants $c_1, c_2>0$ satisfying
$c_1f(n)n^{p/q}\le p_{\X}(n)\le c_2f(n)n^{p/q}$ for all $n\in\N$.
\end{theorem}

\subsubsection{Groups acting on rooted trees}
Let $X$ be a finite alphabet, $|X|\ge 2$. Denote by $X^*$ the set of
all finite words (including the empty word $\varnothing$). We consider
$X^*$ as a rooted tree with root $\varnothing$ in which every word
$v\in X^*$ is connected to the words of the form $vx$ for all $x\in
X$. The \emph{boundary} of the tree is naturally identified with the
space $X^{\N}$ of all one-sided sequences $x_1x_2x_3\ldots$.
Every automorphism of the rooted tree $X^*$ naturally induces a
homeomorphism of $X^{\N}$.

Let $g$ be an automorphism of the tree $X^*$. For every $v\in X^*$
there exists a unique automorphism $g|_v$ of the tree $X^*$ such that
\[g(vw)=g(v)g|_v(w)\]
for all $w\in X^*$. We say that a group $G$ of automorphisms of $X^*$
is \emph{self-similar} if $g|_v\in G$ for every $g\in G$ and $v\in
X^*$. For every $v\in X^*$ and $w\in X^{\N}$ the germ $(g, vw)$
depends only on the quadruple $(v, g(v), g|_v, w)$.

\begin{examp}
\label{ex:admach}
Consider the automorphism $a$ of the binary tree $\{0, 1\}^*$ defined by the recursive rules
\[a(0w)=1w,\qquad a(1w)=0a(w).\]
It is called the \emph{adding machine}, or \emph{odometer}. The cyclic group generated by $a$ is
self-similar.
\end{examp}

\begin{examp}
\label{ex:grigorchuk}
Consider the automorphisms of $\{0, 1\}^*$ defined by the recursive rules
\[a(0w)=1w, a(1w)=0w\]
and
\begin{alignat*}{2}
b(0w)&=0a(w), & \qquad b(1w)&=1c(w),\\
 c(0w)&=0a(w), &\qquad  c(1w)&=1d(w),\\
d(0w)&=0w, &\qquad d(1w)&=1b(w).
\end{alignat*}
The group generated by $a, b, c, d$ is the \emph{Grigorchuk group}, see~\cite{grigorchuk:80_en}.
\end{examp}

For more examples of self-similar groups and their applications, see~\cite{nek:book}.

Let $G$ be a finitely generated self-similar group, and let $l(g)$
denote the length of an element $g\in G$ with respect to some fixed
finite generating set of $G$. The \emph{contraction coefficient} of
the group $G$ is the number
\[\lambda=\limsup_{n\to\infty}\limsup_{g\in G,
  l(g)\to\infty}\max_{v\in X^n}\frac{l(g|_v)}{l(g)}.\]
The group is said to be \emph{contracting} if $\lambda<1$.

For example, the adding machine action of $\Z$ and the Grigorchuk
group are both contracting with contraction coefficient $\lambda=1/2$.

\begin{proposition}
\label{pr:contractingestimates}
Let $G$ be a contracting self-similar group acting on the tree
$X^*$, and let $\lambda$ be the contraction coefficient. Consider the
groupoid of germs $\G$ of the action of $G$ on $X^{\N}$, let $S$
be a finite generating set of $G$, and let $\mathcal{S}$ be the set of
$\G$-bisets of the form $\{(s, w)\;:\;w\in X^{\N}\}$ for $s\in S$.
Then we have
\[\limsup_{n\to\infty}\frac{\log\overline\gamma(n, \mathcal{S})}{\log n}\le
\frac{\log|X|}{-\log\lambda},\qquad
\limsup_{n\to\infty}\frac{\log\delta(n, \mathcal{S})}{\log
  n}\le \frac{\log|X|}{-\log\lambda}.\]
\end{proposition}

\begin{proof}
Let $\rho$ be any number in the interval $(\lambda, 1)$. Then there
exist $n_0$, $l_0$ such that for all elements $g\in G$ such that
$l(g)>l_0$ we have $l(g|_v)\le \rho^{n_0} l(g)$ for all $v\in
X^{n_0}$. It follows that there exists a finite set $\mathcal{N}$
such that $g|_v\in\mathcal{N}$ for all $v\in X^*$ and for every $g\in
G\setminus\mathcal{N}$ we have $l(g|_v)\le \rho^{n_0} l(g)$ for all words
$v\in X^*$ of length at least $n_0$.

Then for every $g\in G$ and for every word $v\in
X^*$ of length at least $\left\lfloor\frac{\log l(g)-\log l_0}{-\log
    \rho}\right\rfloor+n_0$ we
have $g|_v\in\mathcal{N}$. 
Let $w=x_1x_2\ldots\in X^{\N}$, and denote $v=x_1x_2\ldots x_n$,
$w'=x_{n+1}x_{n+2}\ldots$ for $n=\left\lfloor\frac{\log r-\log l_0}{-\log
    \rho}\right\rfloor+n_0$. Then for fixed $w$ and all $g$ such that
$l(g)\le r$, the germ $(g, w)$ depends only on
$g(v)$ and $g|_v$. There are not more than $|X|^n$ possibilities for
$g(v)$, hence the number of germs $(g, w)$ is not more than
\[|\mathcal{N}|\cdot|X|^n\le|\mathcal{N}|\exp\left(\log|X|\left(\frac{\log
      r-\log l_0}{-\log\rho}+n_0\right)\right)\le
C_1r^{\frac{\log|X|}{-\log\rho}}\]
  for $C_1=|\mathcal{N}|\cdot|X|^{\frac{\log
      l_0}{\log\rho}+n_0}$. Consequently, for every $\rho\in (\lambda,
  1)$ there exists $C_1>0$ such that
\[\overline\gamma(r, \mathcal{S})\le C_1r^{\frac{\log|X|}{-\log\rho}},\]
hence $\limsup_{r\to\infty}\frac{\log\overline\gamma(r,
  \mathcal{S})}{\log r}\le\frac{\log|X|}{-\log\lambda}$.

It is enough, in order to know the ball $B_{\mathcal{S}}(w, r)$, to
know for every word $g\in G$ of length at most $2r$ whether the germ
$(g, w)$ is a unit. Let, as above, $w=vw'$, where length of $v$ is
$n=\left\lfloor\frac{\log 2r-\log
    l_0}{-\log\rho}\right\rfloor+n_0$. For every $g\in G$ of
length at most $2r$ the germ $(g, w)$ is a unit if and
only if $g(v)=v$ and $(g|_v, w')$ is a unit. We have
$g|_v\in\mathcal{N}$, so
$B_{\mathcal{S}}(w, r)$ depends only on $v$ and the set
$T_{w'}=\{h\in\mathcal{N}\;:\;(h, w')\in\G^{(0)}\}$. Consequently,
\[\delta(r, \mathcal{S})\le 2^{|\mathcal{N}|}\cdot |X|^n\le
C_2r^{\frac{\log|X|}{-\log\rho}},\]
where $C_2=2^{|\mathcal{N}|}|X|^{\frac{\log l_0-\log 2}{\log\rho}+n_0}$, which shows that
$\limsup_{r\to\infty}\frac{\log\delta(r, \mathcal{S})}{\log r}\le\frac{\log|X|}{-\log\lambda}$.
\end{proof}

Both estimates in Proposition~\ref{pr:contractingestimates} are not
sharp in general. For example, consider a self-similar action of
$\Z^2$ over the alphabet $X$ of size 5 associated with the virtual endomorphism
given by the matrix
$A=\left(\begin{array}{cc} 2 & 1\\ 1 &
    3\end{array}\right)^{-1}=\left(\begin{array}{rr}3/5 & -1/5\\ -1/5 &
        2/5\end{array}\right)$, see~\cite[2.9, 2.12]{nek:book}
    and~\cite{neksid} for details. Note that the eigenvalues
of $A$ are $\left(\frac{5\pm\sqrt{5}}{2}\right)^{-1}\in
    (0, 1)$, hence the contraction coefficient is $\lambda=\frac
    2{5-\sqrt{5}}=\frac{5+\sqrt{5}}{10}$.
On the other hand $\overline\gamma(r, \mathcal{S})$ grows as a
quadratic polynomial, while $\delta(r, \mathcal{S})$ is bounded.

\section{Convolution algebras}

\subsection{Definitions}

Let $\G$ be an \'etale groupoid, and let $\Bbbk$ be a
field. \emph{Support} of a function $f:\G\arr\Bbbk$ is closure of the
set of points $x\in\G$ such that $f(x)\ne 0$. If $f_1, f_2$ are
functions with compact support, then their \emph{convolution} is given
by the formula
\[f_1*f_2(g)=\sum_{h\in\G_{\be(g)}}f_1(gh^{-1})f_2(h).\]
Note that since $f_2$ has compact support, the set of elements
$h\in\G_{\be(g)}$ such that $f_2(h)\ne 0$ is finite.

It is easy to see that if $f_1, f_2$ are supported on the space of
units, then their convolution coincides with their pointwise
product. If $F_1, F_2$ are bisections, then their characteristic
functions satisfy $1_{F_1}*1_{F_2}=1_{F_1F_2}$.

The set of all functions $f:\G\arr\Bbbk$ with compact support forms an
algebra over
$\Bbbk$ with respect to convolution. But this algebra is too big,
and its definition does not use the topology of $\G$ much. On the other
hand, the algebra of all continuous functions (with discrete topology
on $\Bbbk$) is too small in the non-Hausdorff case. Therefore, we
adopt the next definition, following Connes~\cite{conn:foliations}, see also~\cite{paterson:gr} 
and~\cite{steinberg:groupoidapproach}.

\begin{defi}
The \emph{convolution algebra} $\Bbbk[\G]$ is the $\Bbbk$-algebra generated by
the characteristic functions $1_F$ of open compact $\G$-bisections (with
respect to convolution).
\end{defi}

If $\G$ is Hausdorff, then $\Bbbk[\G]$ is the algebra of all
continuous (i.e., locally constant) functions $f:\G\arr\Bbbk$, where
$\Bbbk$ has discrete topology. In the non-Hausdorff case the
algebra $\Bbbk[\G]$ contains discontinuous functions
(e.g., characteristic functions of non-closed open compact
bisections).

From now on we will use the usual multiplication sign for convolution.
The unit of the algebra $\Bbbk[\G]$ is the characteristic function of
$\G^{(0)}$, which we will often denote just by $1$.

If $\G=G\ltimes\X$ is the groupoid of an action, then $\Bbbk[\G]$
is generated by the commutative algebra of locally
constant functions $f:\X\arr\Bbbk$ (with pointwise multiplication and
addition) and the group ring $\Bbbk[G]$ subject to relations
\[g^{-1}\cdot f\cdot g=f\circ g,\]
for all $f:\X\arr\Bbbk$ and $g\in G$, where $f\circ g:\X\arr\Bbbk$ is
given by $(f\circ g)(x)=f(g(x))$.
In other words, it is the \emph{cross-product} of the algebra of
functions and the group ring.
 
Let $\mathcal{T}\subset\G^{(0)}$ be the set of units with trivial
isotropy groups. The set $\mathcal{T}$ is $\G$-invariant, i.e., is a
union of $\G$-orbits.

\begin{defi}
We say that $\G$ is \emph{essentially principal} if the set
$\mathcal{T}$ is dense in $\G^{(0)}$. It is \emph{principal} if
$\mathcal{T}=\G^{(0)}$. The groupoid $\G$ is said to be \emph{minimal}
if every $\G$-orbit is dense in $\G^{(0)}$.
\end{defi}

\begin{examp}
For every homeomorphism $g$ of a metric space $\mathcal{X}$, the set of
points $x\in\mathcal{X}$ such that $g(x)=x$ and the germ $(g, x)$ is
non-trivial is a closed nowhere dense set. It follows that if $G$ is
a countable group of homeomorphisms of $\mathcal{X}$, then groupoid of
germs of the action is essentially principal.
\end{examp}

Simplicity of essentially principal minimal groupoids is a well known fact, 
see~\cite{brownclarketc:simplicity} and a $C^*$-version in~\cite[Proposition~4.6]{renault:groupoids}. 
We provide a proof of the following simple proposition just for completeness.

\begin{proposition}
\label{pr:simple}
Suppose that $\G$ is essentially principal and minimal.
Let $I$ be the set of functions $f\in\Bbbk[\G]$ such that $f(g)=0$ for
every $g\in\G$ such that $\be(g), \en(g)\in\mathcal{T}$. Then $I$ is a
two-sided ideal, and the algebra $\Bbbk[\G]/I$ is simple. In
particular, if $\G$ is Hausdorff, then $\Bbbk[\G]$ is simple.
\end{proposition}

\begin{proof}
The fact that $I$ is a two-sided ideal follows directly from the fact
that $\mathcal{T}$ is $\G$-invariant.

In order to prove simplicity of $\Bbbk[\G]$ it is enough to show that
if $f\in\Bbbk[\G]\setminus I$, then there exist elements $a_i,
b_i\in\Bbbk[\G]$ such that $\sum_{i=1}^ka_ifb_i=1$.

If $f\in\Bbbk[\G]\setminus I$, then there exists $g\in\G$ such that
$\be(g), \en(g)\in\mathcal{T}$ and $f(g)\ne 0$. Let
$f=\sum_{i=1}^m\alpha_i1_{F_i}$, where $F_i$ are open compact
$\G$-bisections. Let $A=\{1\le i\le m\;:\;g\in F_i\}$. Then
$f(g)=\sum_{i\in A}\alpha_i$. Since $\be(g)\in\mathcal{T}$, an
equality of targets $\en(F_i\be(g))=\en(F_j\be(g))$ implies the
equality $F_i\be(g)=F_j\be(g)$ of groupoid elements. It follows that
$\en(F_i\be(g))\ne\en(g)$ for every $i\notin A$. We can find therefore
a clopen neighborhood $U$ of $\be(g)$ such that $U\subset\be(F_i)$, $F_iU=F_jU$, for all
$i, j\in A$, $U\cap\be(F_j)=\emptyset$ for all $j\notin A$,
and $\en(F_iU)\cap\en(F_jU)=\emptyset$ for all $i\in A$ and $j\notin
A$.  Denote $F_iU=F$ for any $i\in A$. 
We have $1_{F^{-1}}f1_U=\sum_{i\in A}\alpha_i 1_U$. It follows that
$1_U=\alpha 1_{F^{-1}}f1_U$ for some $\alpha\in\Bbbk$.

The groupoid $\G$ is minimal, hence for every $x\in\G^{(0)}$ there
exists $h\in\G$ such that $\be(h)=x$ and $\en(h)\in U$. There exists
therefore an open compact $\G$-bisection $H$ such that $x\in\be(H)$ and
$\en(H)\subset U$. Then $1_{\be(H)}=1_{H^{-1}}1_U1_H=\alpha
1_{H^{-1}F^{-1}}f1_U1_H$. It follows that $\G^{(0)}$ can be covered by
a finite collection of sets $V_i$ such that $1_{V_i}$ can be written
in the form $a_ifb_i$ for some $a, b\in\Bbbk[G]$. Note that if
$V_i'$ is a clopen subset of $V_i$, then $1_{V_i'}=1_{V_i'}1_{V_i}$,
hence we may replace the covering $\{V_i\}$ by a finite covering by disjoint
clopen sets. But in that case we have $1=\sum 1_{V_i}$.
\end{proof}

\subsection{Growth of $\Bbbk[\G]$}

\begin{theorem}
\label{th:growth} Let $\G$ be an \'etale groupoid with compact totally
disconnected unit space.
Let $\mathcal{S}$ be a finite set of open compact $\G$-bisections. Let
$V\subset\Bbbk[\G]$ be the $\Bbbk$-subspace generated by the
characteristic functions of the elements of $\mathcal{S}$.
Then
\[\dim V^n\le\overline\gamma(n, \mathcal{S})\delta(n, \mathcal{S}).\]
\end{theorem}

\begin{proof}
Fix $n$, and let $\mathcal{S}^n$ be the set of all products $S_1S_2\ldots S_n$ of length
$n$ of elements of $\mathcal{S}$. Then $V^n$ is the linear span of the characteristic functions of elements
of $\mathcal{S}^n$. Denote, for $x\in\G^{(0)}$,
\[A_x=\bigcap_{F\in\mathcal{S}^n, x\in\be(F)}\be(F)\setminus\bigcup_{F\in\mathcal{S}^n, x\notin\be(F)}\be(F).\]
Since $\be(F)$ is clopen for every $F\in\mathcal{S}^n$, the sets $A_x$ are also clopen.
Note that for every $F\in\mathcal{S}^n$ and $x\in\G^{(0)}$, either
$A_x\subset\be(F)$, or $A_x\cap\be(F)=\emptyset$. 

If $F_1, F_2$ are open $\G$-bisections and
$F_1\cdot x=F_2\cdot x$ for a unit $x$, then the set of points $y$
such that $F_1\cdot y=F_2\cdot y$ is equal to the intersection of
$F_1^{-1}F_2$ with $\G^{(0)}$.
Since $\G$ is \'etale, this set is open.
Denote by $B_x$ the set of all points $y\in A_x$ such that $F_1\cdot
x=F_2\cdot x$ implies $F_1\cdot y=F_2\cdot y$ for all $F_1,
F_2\in\mathcal{S}^n$. Then $B_x$ is open and $x\in B_x$.

Note that if $x\sim_n y$, then $A_x=A_y$, as belonging of a point $y$ to the domain of a product $S_1S_2\ldots S_n$ of elements of $\mathcal{S}$
is equivalent to the existence of a path in $\G(y, \mathcal{S})$ of length $n$ starting at $y$ and labeled by the sequence $S_n, S_{n-1}, \ldots, S_1$.
Similarly, if $x\sim_n y$, then $B_x=B_y$, since an equality $F_1\cdot x=F_2\cdot x$ is equivalent to coincidence of endpoints of the paths corresponding
to the products $F_1$ and $F_2$ starting at $x$.

Let $\mathcal{B}=\{B_x\;:\;x\in\G^{(0}\}$. Since $B_x=B_y$ for $x\sim_n y$, the set $\mathcal{B}$ consists of at most
$\delta(n, \mathcal{S})$ elements.

\begin{lemma}
There exists a covering
$\wt{\mathcal{B}}=\{\wt B\}_{B\in\mathcal{B}}$ of $\G^{(0)}$ by
disjoint clopen sets
such that $\wt B\subset B$ for every $B\in\mathcal{B}$.
\end{lemma}

We allow some of the sets $\wt B$ to be empty.

\begin{proof}
By the Shrinking Lemma, we can find for every $B\in\mathcal{B}$ an
open set $B'\subset
B$ such that $\{B'\}_{B\in\mathcal{B}}$ is a covering of $\G^{(0)}$,
and closure of $B'$ is contained in $B$. Then closure of $B'$ is
compact, and can be covered by a finite collection of clopen subsets
of $B$. Hence, after replacing $B'$ by the union of these clopen
subsets, we may assume that $B'$ are clopen. Order the set
$\mathcal{B}$ into a sequence $B_1, B_2, \ldots, B_m$, define
$\wt B_1=B_1'$, and inductively, $\wt B_i=B_i'\setminus(B_1'\cup
B_2'\cup\cdots\cup B_{i-1}')$. Then $\{\wt B\}_{B\in\mathcal{B}}$
satisfies the conditions of the lemma.
\end{proof}

Let $x_1, x_2, \ldots, x_m$ be a transversal of the $\sim_n$ equivalence relation, where $m=\delta(n, \mathcal{S})$.
For every $F\in\mathcal{S}^n$ and $x_i\in\be(F)$, consider the
restriction $F\cdot \wt B_{x_i}$ of 
$F$ onto $\wt B_{x_i}$. Since $\{\wt B_{x_i}\}_{i=1, \ldots, m}$ is a
covering of $\G^{(0)}$ by disjoint subsets, the sets
$F\cdot\wt B_{x_i}$ form a covering of
$F$ by disjoint subsets, and $1_F=\sum_{i=1}^m1_{F\cdot \wt B_{x_i}}$.

If $F_1, F_2\in\mathcal{S}^n$ and $x_i$ are such that $x_i\in\be(F_1)\cap\be(F_2)$, and $F_1\cdot x_i=F_2\cdot x_i$, then for every
$y\in\wt B_{x_i}$ we have $y\in\be(F_1)\cap\be(F_2)$ and $F_1\cdot y=F_2\cdot y$, hence
$F_1\cdot\wt B_{x_i}=F_2\cdot\wt B_{x_i}$. It follows that $F\cdot\wt B_{x_i}$ depends only on $F\cdot x_i$, and we have not more than $\gamma(n, x_i, \mathcal{S})\le
\overline\gamma(n, \mathcal{S})$ non-empty sets of the form $F\cdot\wt B_{x_i}$, for every given $x_i$. Hence we have at most 
$\overline\gamma(n, \mathcal{S})\delta(n, \mathcal{S})$ functions of the form $1_{F\cdot x_i}$ in total, and every function $1_F$, for
$F\in\mathcal{S}^n$ is equal to the sum of a subset of these functions, which finishes the proof of the theorem.
\end{proof}

\subsection{Finite generation}
For a given finite set $\mathcal{S}$ of open compact $\G$-bisections, generating $\G$, denote
\[A_{x, n}=\bigcap_{F\in\mathcal{S}^n,
  x\in\be(F)}\be(F)\setminus\bigcup_{F\in\mathcal{S}^n,
  x\notin\be(F)}\be(F),\]
see the proof of Theorem~\ref{th:growth}.
Recall that the sets $A_{x, n}$ are clopen. It is also easy to see that two sets $A_{x, n}$ and $A_{y, n}$ are either disjoint
or coincide. Note also that $A_{x, n}\subset A_{x, m}$ if $n>m$. It follows that for any $x, y\in\G^{(0)}$ and $n>m$, either $A_{x, n}\subset A_{y, m}$,
or $A_{x, n}\cap A_{y, m}=\emptyset$.

\begin{defi}
We say that $\mathcal{S}$ is \emph{expansive}
if for any two different points $x, y\in\G^{(0)}$ there exists $n$ such that $A_{x, n}$ and $A_{y, n}$ are disjoint.
\end{defi}

\begin{proposition}
\label{prop:expansive}
If $\mathcal{S}$ is expansive, then the set $\{1_S\;:\;S\in\mathcal{S}\cup\mathcal{S}^{-1}\}$ generates $\Bbbk[\G]$.
\end{proposition}

\begin{proof}
Let $\mathcal{A}$ be the algebra generated by the functions $1_S$ for $S\in\mathcal{S}\cup\mathcal{S}^{-1}$.
Note that $\be(F)=F^{-1}F$, hence $1_F\in A$ for every $F\in(\mathcal{S}\cup\mathcal{S}^{-1})^n$. Note also that $1_{A\cap B}=1_A\cdot 1_B$,
$1_{A\setminus B}=1_A\cdot(1_A-1_B)$, and $1_{A\cup B}=1_A+1_B-1_A1_B$ for every $A, B\subset\G^{(0)}$. It follows that
$1_{A_{x, n}}\in \mathcal{A}$ for all $x\in\G^{(0)}$ and $n$.

Let us show that for every open set $A\subset\G^{(0)}$ and every $x\in A$ there exists $n$ such that $A_{x, n}\subset A$. For every $y\notin A$
there exists $n_y$ such that $A_{x, n_y}\cap A_{y, n_y}=\emptyset$. Since $\G^{(0)}\setminus A$ is compact, there exists a finite
covering $A_{y_1, n_{y_1}}, A_{y_2, n_{y_2}}, \ldots, A_{y_m, n_{y_m}}$ of $\G^{(0)}\setminus A$. Let $n=\max n_{y_i}$. Then $A_{x, n}\subset A$.

Let $F$ be an arbitrary open compact $\G$-bisection. For every $g\in F$ there exists $n$ and $F'\in(\mathcal{S}\cup\mathcal{S}^{-1})^n$
such that $g\in F'$. There also exists $n_g$ such that $A_{\be(g), n_g}\subset\be(F)$ and $F\cdot A_{\be(g), n_g}=F'\cdot A_{\be(g), n_g}$.
We get a covering of $F$ by sets of the form $F'\cdot A_{x, m}$, where $F'\in(\mathcal{S}\cup\mathcal{S}^{-1})^n$. Since any two
sets of the form $A_{x, n}$ are either disjoint or one is a subset of the other, we can find a covering of $F$ by disjoint sets of the form
$F'\cdot A_{x, m}$ for $F'\in(\mathcal{S}\cup\mathcal{S}^{-1})^n$. This implies that $1_F\in\mathcal{A}$, which finishes the proof.
\end{proof}

\subsection{Examples}

\subsubsection{Subshifts}
\label{sss:subshifts}
Let $\X\subset X^{\Z}$ be a subshift, and let
$\mathfrak{S}$ be the groupoid of germs generated by the shift
$s:\X\arr\X$. Let, as in~\ref{sss:shifts}, $S_x=\{(s, w)\;:\;w(0)=x\}$,
$\mathcal{S}=\{S_x\}_{x\in X}$. Note that
for every word $x_1x_2\ldots x_n$ domain of the product
$S_{x_1}S_{x_2}\cdots S_{x_n}$ is the set of words $w\in\X$ such that
$w(0)=x_n$, $w(1)=x_{n-1}$, \ldots, $w(n-1)=x_1$. It follows that the
set $\mathcal{S}\cup\mathcal{S}^{-1}$ is expansive, and by
Proposition~\ref{prop:expansive},
$\{1_S\}_{S\in\mathcal{S}\cup\mathcal{S}^{-1}}$ is a generating set of
$\Bbbk[\mathfrak{S}]$. 

Since $\mathfrak{S}$ coincides with the groupoid of the
$\Z$-action on $\X$ defined by the shift, the algebra $\Bbbk[\mathfrak{S}]$ is
the corresponding cross-product of the algebra of continuous $\Bbbk$-valued
functions with the group algebra of $\Z$. Every its element is
uniquely written as a Laurent polynomial $\sum a_n\cdot t^n$, where $t\in\Bbbk[\G]$ is
the characteristic function of the set of germs of the shift $s:\X\arr\X$, and
$a_n$ are continuous $\Bbbk$-valued functions. Multiplication rule for such
polynomials follows from the relations $t\cdot a=b\cdot t$, where
$a, b:\X\arr\Bbbk$ satisfy $b(w)=a(s^{-1}(w))$ for every $w\in\X$.

\begin{proposition}
\label{pr:shiftgrowth}
Let $V$ be the linear span of
$\{1\}\cup\{1_S\}_{S\in\mathcal{S}\cup\mathcal{S}^{-1}}$. Then
\[\left\lfloor\frac{n}2\right\rfloor
p_{\X}\left(\left\lfloor\frac n2\right\rfloor\right)\le\dim V^n\le (2n+1)p_{\X}(2n).\]
\end{proposition}

\begin{proof}
The upper bound follows from Theorem~\ref{th:growth}. For the lower
bound note that $S_{x_1}S_{x_2}\ldots S_{x_n}$ and
$S_{y_1}S_{y_2}\ldots S_{y_m}$ are disjoint if $x_1x_2\ldots x_n\ne
y_1y_2\ldots y_m$, hence the set of characteristic functions of all
non-zero products of elements of $\mathcal{S}$ is linearly
independent, so that $\sum_{k=0}^np_{\X}(n)\le\dim V^n$. Since
$p_{\X}(n)$ is non-decreasing, we have
$\left\lfloor\frac{n}2\right\rfloor
p_{\X}\left(\left\lfloor\frac n2\right\rfloor\right)\le\sum_{k=0}^np_{\X}(n)$.
\end{proof}

Note that since the characteristic functions of the
products $S_{x_1}S_{x_2}\ldots S_{x_n}$ are linearly independent,
their linear span is a sub-algebra of $\Bbbk[\mathfrak{S}]$ isomorphic
to the semigroup algebra $\mathcal{M}_{\X}$ 
of the semigroup generated by the set
$\{S_x\;:\;x\in X\}$. It is easy to see that $\mathcal{M}_{\X}$ is
isomorphic to the quotient of the free associative algebra generated
by $X$ modulo the ideal generated by all words $w\in X^*$ such that
$w$ is not a subword of any element of the subshift $\X$. It follows
from Proposition~\ref{pr:shiftgrowth} that growths of
$\Bbbk[\mathfrak{S}]$ and $\mathcal{M}_{\X}$ are equivalent. Note that
the algebras $\mathcal{M}_{\X}$ are the original examples of algebras
of arbitrary Gelfand-Kirillov dimension, see~\cite{warfield:gk}
and~\cite[Theorem~2.9]{krauselenagan}.

\begin{examp}
Let $\X$ be a Sturmian subshift. It is minimal
and $p_{\X}(n)=n+1$, hence \[\frac{(n+1)(n+2)}2\le\dim
V^n\le 2n(2n+1),\]
so that $\Bbbk[\mathfrak{S}]$ is a quadratically growing finitely
generated algebra. Note that it is simple by Proposition~\ref{pr:simple}. This disproves
Conjecture~3.1 in~\cite{bell:simple}.
\end{examp}

\begin{examp}
It is easy to see that every Toeplitz subshift is
minimal. Consequently, known examples of Toeplitz subshifts
(see Subsection~\ref{sss:shifts}) provide us with simple finitely
generated algebras of arbitrary Gelfand-Kirillov dimension $\alpha\ge
2$, and also uncountably many different growth types of simple
finitely generated algebras of
Gelfand-Kirillov dimension two (see a question on existence of such
algebras on page 832 of~\cite{bell:dichotomy}).
\end{examp}

\subsubsection{Self-similar groups}

Let $G$ be a self-similar group of automorphisms of the tree
$X^*$. Let $\G$ be the groupoid of germs of its action on the boundary
$X^{\N}$ of the tree. Suppose that $G$ is \emph{self-replicating},
i.e., for all $x, y\in X$ and $g\in G$ there exists $h\in G$ such that
$g(x)=y$ and $h|_x=g$. Then for all pairs of words $v, u\in X^*$ of equal
length and every $g\in G$ there exists $h\in G$ such that $h(v)=u$ and
$h|_v=g$. In other words, the transformation $vw\mapsto ug(w)$ is an
open compact $\G$-bisection (more pedantically, the set of its germs
is a bisection, but we will identify a $\G$ bisection $F$ with the map
$\be(g)\mapsto\en(g)$, $g\in F$).

Fix $n\ge 0$, and consider the set of all $\G$-bisections of the form
$R_{u, g, v}:vw\mapsto ug(w)$ for $v, u\in X^n$ and $g\in G$. Note
that these bisections are multiplied by the rule
\begin{equation}\label{eq:Rugv}
R_{u_1, g_1, v_1}R_{u_2, g_2, v_2}=\left\{\begin{array}{ll} 0 &
    \text{if $v_1\ne u_2$;}\\ R_{u_1, g_1g_2, v_2} & \text{if $v_1=u_2$}.\end{array}\right.
\end{equation}
Let $A_n$ be the formal linear span of the elements $R_{u, g, v}$ for $u, v\in X^n$ and $g\in G$. Extend
multiplication rule~\eqref{eq:Rugv} to $A_n$. It is easy to see then that $A_n$ is isomorphic to the algebra
$M_{d^n\times d^n}(\Bbbk[G])$ of matrices of size $d^n\times d^n$ over the group ring $\Bbbk[G]$.

The map $R_{u, g, v}\mapsto\sum_{x\in X} R_{ug(x), g|_x, vx}$ induces a homomorphism $A_n\mapsto A_{n+1}$ called
the \emph{matrix recursion}. More on matrix recursions for
self-similar groups
see~\cite{bgr:spec,bartholdi:ring,nek:bim,nek:cpalg,grinek:schur}.

\begin{examp}
For the adding machine action (see Example~\ref{ex:admach}) the matrix recursions replace every entry $a^n$ by $\left(\begin{array}{cc} 0 & a\\
1 & 0\end{array}\right)^n$, i.e., are induced by the map
\[a\mapsto\left(\begin{array}{cc} 0 & a\\ 1 & 0\end{array}\right).\] For
example, the image of $a$ in $A_2$ is 
\[\left(\begin{array}{cccc}0 & 0 & 0 & a\\ 0 & 0 & 1 & 0\\ 1 & 0 & 0 &
    0\\ 0 & 1 & 0 & 0\end{array}\right).\]

For the Grigorchuk group the matrix recursions are induced by the map 
\[a\mapsto\left(\begin{array}{cc} 0 & 1\\ 1 & 0\end{array}\right),\qquad 
b\mapsto\left(\begin{array}{cc} a & 0\\ 0 & c\end{array}\right),\] 
\[c\mapsto\left(\begin{array}{cc} a & 0\\ 0 & d\end{array}\right),\qquad
d\mapsto\left(\begin{array}{cc} 1 & 0\\ 0 & b\end{array}\right).\]
\end{examp}

\begin{proposition}
The convolution algebra $\Bbbk[\G]$ of the groupoid of germs of the action of $G$ on $X^{\N}$ is isomorphic to the direct
limit of the matrix algebras $A_n\cong M_{d^n\times d^n}(\Bbbk[G])$ with respect to the matrix recursions.
\end{proposition}

\begin{proof}
Denote by $A_\infty$ the direct limit of the algebras $A_n$ with respect to the matrix recursions. Let $\phi:A_\infty\arr\Bbbk[\G]$ be
the natural map given by $\phi(R_{u, g, v})=1_{R_{u, g, v}}$. Note 
that $1_{R_{u, g, v}}=\sum_{x\in X}1_{R_{ug(x), g|_x, vx}}$, hence
the map $\phi$ is well defined. It also follows from equation~\eqref{eq:Rugv} that $\phi$ is a homomorphism of algebras. It remains to show that
$\phi$ is injective. Let $f$ be a non-zero element of $\Bbbk[\G]$, and let $(g, w)\in\G$ be such that $f(g, w)\ne 0$. Suppose that
$\phi(f)=\sum_{u, v\in X^n}\alpha_{u, v}R_{u, g_{u, v}, v}$ for some $\alpha_{u, v}\in\Bbbk$ and $g_{u, v}\in G$.
Denote the set
of all pairs $(u, v)$ such that $(g, w)\in R_{u, g_{u, v}, v}$ and $\alpha_{u, v}\ne 0$ by $P$. 
The set $\bigcap_{(u, v)\in P}R_{u, g_{u, v}, v}$ is an open neighborhood of $(g, w)$, hence there exists a $\G$-bisection
$R_{w_1, h, w_2}$ contained in $\bigcap_{(u, v)\in P}R_{u, g_{u, v}, v}$. Applying the matrix recursion, we get a representation 
of $f$ as an element $\sum_{u, v\in X^{|w_1|}}\beta_{u, v}R_{u, h_{u, v}, v}\in A_{|w_1|}$ such that $(g, w)$ does not belong to 
any set $R_{u, h_{u, v}, v}$, $u, v\in X^{|w_1|}$, $(u, v)\ne (w_1, w_2)$. Then $f(g, w)=\beta_{u, v}\ne 0$, hence $\phi(f)\ne 0$.
\end{proof}

As a corollary of Proposition~\ref{pr:contractingestimates}
and Theorem~\ref{th:growth} we get the following result of L.~Bartholdi~\cite{bartholdi:ring}.

\begin{proposition}
Let $G$ be a contracting self-replicating group, and let $\G$ be the groupoid of germs
of its action on $X^{\N}$. Every finitely generated sub-algebra of $\Bbbk[\G]$ has Gelfand-Kirillov dimension at most
$\frac{2\log|X|}{-\log\lambda}$, where $\lambda$ is the contraction coefficient of $G$.
\end{proposition}

The image of the group ring $\Bbbk[G]$ in $\Bbbk[\G]$ is called the
\emph{thinned algebra}. It was defined in~\cite{sid:ring}, see also~\cite{bartholdi:ring}. 

Let us come back to the case of the Grigorchuk group. Since its contraction coefficient is equal to $1/2$, every finitely generated sub-algebra
of $\Bbbk[\G]$ has Gelfand-Kirillov dimension at most 2. It is easy to prove that it is actually equal to 2 in this case. Moreover, it
has quadratic growth, see~\cite{bartholdi:ring}.

This example is also an illustration of the non-Hausdorffness phenomenon.
The groupoid of germs of the Grigorchuk group is not Hausdorff: the germs $(b, 111\ldots)$, $(c, 111\ldots)$, $(d, 111\ldots)$, 
and $(1, 111\ldots)$ do not have disjoint neighborhoods. 

\begin{examp}
\label{ex:nonHausdorff}
Consider the convolution algebra $\mathbb{F}_2[\G]$ for the groupoid
of germs of the Grigorchuk group over the field with two elements.
The matrix recursion for the element $b+c+d+1$ is
\[b+c+d+1\mapsto\left(\begin{array}{cc}0 & 0 \\ 0 &
    b+c+d\end{array}\right).\] It follows that
$b+c+d$ is a non-trivial element of $\mathbb{F}_2[\G]$ but, as a
function on $\G$ is zero everywhere except for the germs of $b, c, d,
1$ at $111\ldots$,
where it is equal to 1. This shows that the ideal $I$ from
Proposition~\ref{pr:simple} is non-zero in this case, and the algebra
$\mathbb{F}_2[\G]$ is not simple.
\end{examp}

\subsection{Modules $\Bbbk\G_x$}

Let $\G$ be an \'etale minimal groupoid. Consider the
space $\Bbbk\G_x$ of maps $\phi:\G_x\arr\Bbbk$ with finite support, where
$\G_x=\{g\in\G\;:\;\be(g)=x\}$. It is easy to see that for every
$\phi\in\Bbbk\G_x$ and $f\in\Bbbk[\G]$ the convolution
$f\cdot\phi$ is an element of
$\Bbbk\G_x$, and that $\Bbbk\G_x$ is a left $\Bbbk[\G]$-module with
respect to the convolution.

\begin{proposition}
\label{pr:modules}
Let $\mathcal{S}$ be
an finite set of open compact $\G$-bisections, and let $V\subset\Bbbk[\G]$ be the
linear span of their characteristic functions and $1_{\G^{(0)}}$. Then
for every $n\ge 1$ we have
\[\dim V^n\cdot\delta_x\le\gamma_{\mathcal{S}}(x, n),\]
where $\delta_x\in\Bbbk\G_x$ is the characteristic function of
$x\in\G_x$, and $\gamma_{\mathcal{S}}(x, n)$ is the growth of the
Cayley graph based at $x$ of the groupoid generated by the union of
the elements of $\mathcal{S}$.

If the isotropy group of $x$ is trivial, then the module $\Bbbk\G_x$
is simple.
\end{proposition}

\begin{proof}
The growth estimate is obvious, since for every $g\in\G_x$ and
$S\in\mathcal{S}$ we have $1_S\cdot\delta_g=\delta_{Sg}$, if
$Sg\ne\emptyset$, and $1_S\cdot\delta_g=0$ otherwise.

Let us show that $\Bbbk\G_x$ is simple if the isotropy group of $x$ is
trivial. It is enough to show that for
every non-zero element
$\phi\in\Bbbk\G_x$ there exist elements $f_1, f_2\in\Bbbk[\G]$
such that $f_1\cdot\phi=\delta_x$ and $f_2\cdot\delta_x=\phi$.

Let $\phi\in\Bbbk\G_x$, and let $\{g_1, g_2, \ldots, g_k\}$ be the
support of $\phi$. Since the isotropy group of $x$ is trivial,
$\en(g_i)$ are pairwise different. Let $U_1, U_2, \ldots, U_k$ be open
compact $\G$-bisections such that $g_i\in U_i$ and $\en(U_i)$ are
disjoint. Then $\left(\sum_{i=1}^k\phi(g_i)1_{U_i}\right)\cdot
\delta_x=\phi$ and $\phi(g_1)^{-1}1_{U_1^{-1}}\phi=\delta_x$.
\end{proof}

\begin{examp}
\label{ex:matrices}
Let $X$ be a finite alphabet, and let $w\in X^{\Z}$ be a non-periodic
sequence such that closure $\X_w$ of the shift orbit of $w$ is
minimal.  Let $\mathfrak{S}$
be the groupoid generated by the action of the shift on $\X_w$. Denote
by $T$ and $T^{-1}$ the characteristic functions of the sets of germs
of the shift and its
inverse, and for every $x\in X$, denote by $D_x$ the characteristic
function of the cylindrical set $\{w\in\X_w\;:\;w(0)=x\}$. Then
$\Bbbk[\mathcal{S}]$ is generated by $T, T^{-1}$ and $D_x$ for $x\in
X$. Note that we can remove one of the generators $D_x$, since
$\sum_{x\in X}D_x=1=TT^{-1}$. Consider the set
$\mathfrak{S}_w=\{(s^n, w)\;:\;n\in\Z\}$ and the corresponding module
$\Bbbk\mathfrak{S}_w$. Its basis as a $\Bbbk$-vector space consists of
the delta-functions $e_n=\delta_{(s^n, w)}$, $n\in\Z$. In this
naturally ordered basis left
multiplication by $T$ is given by the
matrix \[T=\left(\begin{array}{cccccc}\ddots & \vdots & \vdots &
    \vdots & \vdots &  \\ \cdots & 0 & 0 & 0 & 0 & \cdots \\ \cdots & 1 & 0 &
    0 & 0 & \cdots \\ \cdots & 0 & 1 & 0 & 0 & \cdots \\  
\cdots & 0 & 0 & 1 & 0 & \cdots \\
& \vdots &
    \vdots &
    \vdots & \vdots & \ddots\end{array}\right)=(t_{ij})_{i\in\Z,
  j\in\Z}\]
with the entries $t_{m, n}=\delta_{m-1, n}$. The element $T^{-1}$ is given
by the transposed matrix, and an element $D_x$ is given by the
diagonal matrix $(a_{ij})$ with entries given by the rule
\[a_{nn}=\left\{\begin{array}{ll}1 & \text{if $w(n)=x$,}\\ 0 &
    \text{otherwise.}\end{array}\right.\]
It follows that the algebra $\Bbbk[\mathfrak{S}]$ is isomorphic to the
algebra generated by such matrices. For example, if $X=\{0, 1\}$, then
the algebra is generated by the matrices $T$, $T^\top$, and the
diagonal matrix with the sequence $w$ on the diagonal.
\end{examp}

\def\cprime{$'$}\def\ocirc#1{\ifmmode\setbox0=\hbox{$#1$}\dimen0=\ht0
  \advance\dimen0 by1pt\rlap{\hbox to\wd0{\hss\raise\dimen0
  \hbox{\hskip.2em$\scriptscriptstyle\circ$}\hss}}#1\else {\accent"17 #1}\fi}
  \def\cprime{$'$}

\end{document}